\documentclass{amsart}

\usepackage{amsmath}
\usepackage{amsthm}
\usepackage{amssymb}
\usepackage{amsfonts}
\usepackage{enumitem}
\usepackage{setspace}
\usepackage[bookmarks=true,hyperindex, pdftex,colorlinks,citecolor=blue]{hyperref}
\usepackage{centernot} 
\usepackage{marginnote} 
\usepackage[left=3cm,right=3cm,top=2cm,bottom=2cm]{geometry}
\usepackage{bbm}
\usepackage[all]{xy}
\usepackage{tikz}
\usetikzlibrary{arrows}
\usetikzlibrary{shapes,decorations}
\usetikzlibrary{positioning}

\def\<{\langle}
\def\>{\rangle}

\theoremstyle{plain}
\newtheorem{theorem}{Theorem}[section]
\newtheorem{lemma}[theorem]{Lemma}
\newtheorem{proposition}[theorem]{Proposition}

\theoremstyle{definition}

\theoremstyle{remark}
\newtheorem{remark}[theorem]{Remark}

\numberwithin{equation}{section}
\usetikzlibrary{matrix}
\usepackage[all]{xy}
\begin{document}

\title[Recurrence of multiples of composition operators on weighted Dirichlet spaces]{Recurrence of multiples of composition operators on weighted Dirichlet spaces}

\author[N. Karim]{Noureddine Karim}

\author[O. Benchiheb]{Otmane Benchiheb}

\author[M. Amouch]{Mohamed Amouch}
\address[Noureddine Karim, Otmane Benchiheb, and Mohamed Amouch]{Chouaib Doukkali University.
Department of Mathematics, Faculty of science
El Jadida, Morocco}
\email{noureddinekarim1894@gmail.com}
\email{otmane.benchiheb@gmail.com}
\email{amouch.m@ucd.ac.ma}


\date{} 


\begin{abstract}
A bounded linear operator $T$ acting on a Hilbert space $\mathcal{H}$ is said to be recurrent if for every non-empty open subset $U\subset \mathcal{H}$ there is an integer $n$ such that $T^n (U)\cap U\neq\emptyset$. In this paper, we completely characterize the recurrence of scalar multiples of composition operators, induced by linear fractional self maps of the unit disk, acting on weighted Dirichlet spaces $S_\nu$; in particular on the Bergman space, the Hardy space, and the Dirichlet space. Consequently, we complete a previous work of Costakis et al. \cite{costakis} on recurrence of linear fractional composition operators on Hardy space. In this manner, we determine the triples $(\lambda,\nu,\phi)\in \mathbb{C}\times \mathbb{R}\times LFM(\mathbb{D})$ for which the scalar multiple of composition operator $\lambda C_\phi$ acting on $S_\nu$ fails to be recurrent.
\end{abstract}

\renewcommand\thetable{\Roman{table}}

\maketitle
\section{Introduction and preliminaries}
Throughout this paper, $\mathbb{C}$ will represent the complex plane, $\mathbb{C}^*$ the punctured plane $\mathbb{C}\backslash\{0\}$, and $\hat{\mathbb{C}} = \mathbb{C}\cup \{\infty\}$ will be the one-point compactification of $\mathbb{C}$. Moreover, $\mathbb{D}$ will stand for the open unit disk of $\mathbb{C}$, while $\mathbb{T}$ will represent the unit circle of $\mathbb{C}$.

A bounded linear operator $T$ acting on a Hilbert space $\mathcal{H}$ is said to be hypercyclic if there is a vector $f\in \mathcal{H}$ whose $T$-orbit;
$$\mathcal{O}(f,T)=\{T^n f;\ n\in \mathbb{N}\},$$
is dense in $\mathcal{H}$. In such a case the vector $f$ is called a hypercyclic vector.
An operator $T$ is said to be cyclic if there is a vector $f\in \mathcal{H}$ such that the space generated by $\mathcal{O}(f,T)$;
$$\textrm{span}(\mathcal{O}(f,T))=\{p(T)f;\ p\ \textrm{a polynomial}\},$$
is dense in $\mathcal{H}$.
In this case the vector $f$ is called a cyclic vector.
A strong form of cyclicity and weaker than hypercyclicity is supercyclicity.
An operator $T$ is said to be supercyclic if there is a vector $f\in \mathcal{H}$ whose projective orbit;
$$\mathbb{C}.\mathcal{O}(f,T)=\{\lambda T^n f;\ n\in \mathbb{N}\ \textrm{and}\ \lambda\in \mathbb{C} \},$$
is dense in $\mathcal{H}$.
The vector $f$ is called a supercyclic vector.


Recall \cite{universal}, that an operator $T$ on a Hilbert space $\mathcal{H}$ is hypercyclic if and only if it is topologically transitive; that is, for any pair of non-empty open subsets $U$, $V$ of
$\mathcal{H}$ there exists some $n \in \mathbb{N}$ such that
$$T^n (U)\cap V\neq\emptyset.$$
Recently, the notions of hypercyclic and transitivity has been generalized and studied see \cite{recpro,uppfre,strtra,onthespe}.\\
Another important concept in topological dynamics is that of recurrence. This notion has been initiated by Poincar\'{e} and Birkhoff, while a systematic study was given by Costakis et al. in \cite{costakis}. A bounded linear operator on a Hilbert space $\mathcal{H}$ is called recurrent if for every non-empty open subset $U\subset \mathcal{H}$ there is a positive integer $n$ such that
$$T^n (U)\cap U\neq\emptyset.$$
A vector $f\in \mathcal{H}$ is said to be recurrent for $T$ if there exists a strictly increasing sequence of positive integers $(n_k)_{k\in \mathbb{N}}$ such that
$$T^{n_k}f\rightarrow f,\\ \mbox{ as } \ \ k\rightarrow \infty.$$
In that sense, every hypercyclic operator is recurrent, and every hypercyclic vector is recurrent.
For more information on linear dynamics we refer to \cite{lincha} and \cite{dynoflinope}.

The study of linear dynamics has become a very active area
of research. This work will be devoted to studying the recurrence of composition operators.
Recall that if $\mathcal{H}$ is a Hilbert space of analytic functions in the unit disk $\mathbb{D}$, and if $\phi$ is a nonconstant self map of $\mathbb{D}$,
then the composition operator $C_\phi$ associated to $\phi$ on $\mathcal{H}$ is defined by
$$C_\phi f=f\circ \phi \ \ \mbox{ for all } \ \ f \in \mathcal{H}.$$
In which case, the function $\phi$ called a symbol of $C_\phi$. For general references on the theory of composition operators, see, e.g., Cowen and MacCluer's book \cite{CM}, Shapiro's book \cite{Sh} and K. H. Zhu's book \cite{spaholfun}. The special about the composition operator is that the properties of $C_\phi$ depend significantly on the behaviour of the symbol $\phi$. In this paper, we show that the recurrence of the composition operator is influenced by the location of the fixed points of its symbol.

For each real number $\nu$ the weighted Dirichlet space $\mathcal{S}_\nu$ is the space of functions $f(z)=\sum_{n=0}^{\infty}a_n z^n$ analytic on $\mathbb{D}$ such that the following norm
\begin{equation}\label{21}
\|f\|_\nu^2=\sum_{n=0}^{\infty}|a_n|^2(n+1)^{2\nu}
\end{equation}
is finite. Endowed with the inner product
$$\left\langle  \sum_{n=0}^{\infty}a_n z^n,\sum_{n=0}^{\infty}b_n z^n\right\rangle =\sum_{n=0}^{\infty}a_n \bar{b}_n(n+1)^{2\nu},$$
the spaces $\mathcal{S}_\nu$ are Hilbert spaces, see \cite[p. 16]{CM} or \cite[p. 1]{EM}.
For some values of $\nu$ the spaces $\mathcal{S}_\nu$ are very well known classical analytic
function spaces: for instance if $\nu = 1/2$, $S_\nu$ is the Dirichlet space $\mathcal{D}$; for $\nu = 0$ it is the
Hardy space $H^2$ and for $\nu = -1/2$ it is the Bergman space $\mathcal{A}^2$. Observe that if $\nu_1>\nu_2$, then the space $\mathcal{S}_{\nu_1}$ is strictly contained in $\mathcal{S}_{\nu_2}$, and that $\|f\|_{\nu_2}\leq\|f\|_{\nu_1}$, for every $f\in \mathcal{S}_{\nu_1}$.

Also, we can define the Dirichlet space as the collection of functions analytic on $\mathbb{D}$ whose first derivatives have square integrable modulus over $\mathbb{D}$. For $f\in \mathcal{D}$ the norm in $\mathcal{D}$ has the representation
$$\|f\|_{\mathcal{D}}^2=|f(0)|^2+\int_{\mathbb{D}}|f'(z)|^2dA(z),$$
where here $dA(z)$ is the Lebesgue area measure on $\mathbb{D}$ normalized to have unit mass. In the Hardy space $H^2$ there is also an integral representation of the norm. This representation is the following
$$\|f\|_{H^2}^2=\frac{1}{2\pi}\sup_{0<r<1}\int_{-\pi}^{\pi}|f(re^{i\theta})|^2d\theta.$$
The Dirichlet space and the Hardy space will play a central role in our study.

Let $\nu\in \mathbb{R}$, we define the function on $\mathbb{D}$ as: $$k(z)=\sum_{n=0}^{\infty}\frac{z^n}{(n+1)^{2\nu}},$$ which is analytic on $\mathbb{D}$. Then for each $w\in \mathbb{D}$, the reproducing kernel is defined by $$K_w(z)=k(\bar{w}z).$$ This is easily seen that $\|K_w\|^2=k(|w|^2)$ and for every $f(z)=\sum_{n=0}^{\infty}a_nz^n\in S_\nu$ we have that $$\langle f,K_w\rangle=\sum_{n=0}^{\infty}a_nw^n=f(w).$$

It is known according to a result of P.R. Hurst \cite{PH} that the composition operator $C_\phi$ is always bounded on $S_\nu$ when $\phi$ is a linear fractional map on $\mathbb{D}$.\\
Recall (see, e.g., \cite[Chapter 3]{Ah}, \cite[Chapter 0]{Sh}) that linear fractional maps are those maps of the form $$\phi(z) = \frac{az + b}{cz + d},$$ where $a,b,c$ and $d$ are complex numbers satisfying $ ad - bc\neq0$. They extend to the extended complex plane $\mathbb{\hat{C}}$ by defining $\phi(\infty)=a/c$, and $\phi(-d/c)=\infty$ if $c\neq0$, while $\phi(\infty)=\infty$ if $c=0$. The linear fractional maps can be classified according to their fixed point, which are at most two fixed points if $\phi$ is not the identity. Two linear fractional maps $\phi$ and $\psi$ are called conjugate if there is another linear fractional map $T$ such that $$\phi=T^{-1}\psi T.$$ It is called parabolic if $\phi$ has only one fixed point $\alpha$ or, equivalently, it is conjugate to a translation $\psi(z)=z+\tau$, were $\tau\neq0$. If $\phi$ has two distinct fixed points $\alpha$ and $\beta$, then $\phi$ is conjugate to a dilation $\psi(z)=\mu z$. In this case, the linear fractional map $\phi$ is called elliptic if $|\mu|=1$; hyperbolic if $\mu>0$ and loxodromic otherwise, see \cite{Ah} for more details. The chain role proves that the value of the derivative at the fixed point is preserved under conjugation. Therefore, the derivative of a parabolic linear fractional map at its fixed point is $1$, while the derivative of a hyperbolic one is strictly less than $1$ at its attractive fixed point and greater than $1$ at its repulsive point. We will note $LFM(\mathbb{D})$ to refer to the set of all such maps,
which are additionally self-maps of the unit disk $\mathbb{D}$. It is known that if $\phi\in LFM(\mathbb{D})$, then the derivative $\phi'(\eta)$ exists and is finite for every $\eta$ in the unit circle $\mathbb{T}$. The condition $\phi(\mathbb{D})\subset \mathbb{D}$ make some exigences on the location of the fixed points of $\phi$. We have that (see \cite[p. 5]{Sh})
\begin{enumerate}
  \item If $\phi$ is parabolic, then its fixed point $\eta$ is in $\mathbb{T}$ and it satisfies $\phi'(\eta)=1$.
  \item If $\phi$ is hyperbolic, the attractive fixed point is in $\bar{\mathbb{D}}$ and the other fixed point outside of $\mathbb{D}$ and both fixed points are on $\mathbb{T}$ if and only if $\phi$ is an automorphism of $\mathbb{D}$.
  \item If $\phi$ is loxodromic or elliptic, one fixed point is in $\mathbb{D}$ and the other fixed point outside of $\bar{\mathbb{D}}$. The elliptic ones are always automorphisms of $\mathbb{D}$. The fixed point in $\mathbb{D}$ for the loxodromic ones is attractive.
\end{enumerate}
%

Concerning the dynamics of composition operator on Dirichlet spaces, many works have been realized:

In her dissertation \cite{Zo}, Nina Zorboska studied composition operators induced by non-elliptic disk automorphism and proved that they are all cyclic on the Hardy space. In addition, in \cite{zo3} she has obtained some results about cyclicity and hypercyclicity on the so-called smooth weighted Hardy spaces, which are the weighted Hardy spaces whose functions have continuous first derivatives on the boundary of the unit disk (basically weighted Hardy spaces strictly smaller than $\mathcal{S}_{3/2}$).

In \cite{BS1,BS2} Bourdon and Shapiro thoroughly characterized the cyclicity, the supercyclicity, and the hypercyclicity of linear fractional composition operators on the Hardy space $H^2(\mathbb{D})$ (see Table I). Moreover, in \cite{BS2} they gave a program of transferring the cyclic and hypercyclic properties of linear fractional composition operators to general composition operators on the Hardy space.

In \cite{EM} answering some open questions posed by Zorboska \cite{zo3},  Gallardo-Guti\'{e}rrez and Montes-Rodr\'{\i}guez obtained a complete characterization of the cyclic and hypercyclic composition operators induced by linear fractional maps on weighted Dirichlet spaces $\mathcal{S}_\nu$ (see Table II).

In the present work, we give a complete characterization of recurrence of scalar multiples of composition operator whose symbols are linear fractional maps, acting on weighted Dirichlet spaces $\mathcal{S}_\nu$; in particular, on the Bergman space $\mathcal{S}_{-1/2}$, the Hardy space $\mathcal{S}_0$ and the classical Dirichlet space $\mathcal{S}_{1/2}$. To do that, spectral technics will play a significant role.

The paper is organized as follows$:$
In the third section, we characterize the recurrence of linear fractional composition operator $C_\phi$ on $\mathcal{S}_\nu$ by providing a necessary and sufficient condition on $\phi$ and the real number $\nu$.

The fourth section is devoted to studying the recurrence of scalar multiples of linear fractional composition operators $\lambda C_\phi$ on the weighted Dirichlet spaces by providing a necessary and sufficient condition on $\phi$, the real $\nu$, and the complex scalar $\lambda$.

In tables I and II, we expose the classification of Gallardo-Guti\'{e}rrez and Montes-Rodr\'{\i}guez of the dynamics of composition operators and their scalar multiples on $\mathcal{S}_\nu$. In table III, we summarize our results about the recurrence of linear fractional composition operator on $\mathcal{S}_\nu$. Finally, in table IV, we give a summarize of the recurrence of scalar multiples of linear fractional composition operator on $\mathcal{S}_\nu$.
\begin{table}[h!]
\centering
\begin{tabular}{|p{3cm}|c|c|c|}
  \hline
  Type of $\phi$ & Cyclic & Supercyclic & Hypercyclic \\
  \hline\hline
   Hyperbolic Automorphism & $\nu<1/2$ & $\nu<1/2$ & $\nu<1/2$ \\
   \hline
  Parabolic Automorphism & $\nu<1/2$ & $\nu<1/2$ & $\nu<1/2$ \\
  \hline
  Hyperbolic Non-Automorphism & Always & $\nu\leq1/2$ & $\nu<1/2$ \\
  \hline
   Parabolic Non-Automorphism& $\nu\leq3/2$ & Never & Never \\
   \hline
   Interior and Exterior& Always & Never & Never \\
   \hline
   Interior and Boundary& Never & Never & Never \\
   \hline
   Elliptic Irrational rotation& Always & Never & Never \\
   \hline
   Elliptic Rational rotation& Never & Never & Never \\
  \hline
\end{tabular}
\caption{Cyclic properties of $C_\phi$ on $\mathcal{S}_\nu$, $\phi$ linear fractional map.}
\end{table}

\begin{table}[h!]
\begin{tabular}{|p{3cm}|c|}
  \hline
   Type of $\phi$& Hypercyclicity \\
   \hline\hline
   Hyperbolic Automorphism& $\nu<1/2$ and $\phi'(\eta)^{(1-2\nu)/2}<|\lambda|<\phi'(\eta)^{(2\nu-1)/2}$ \\
   \hline
   Parabolic automorphism& $\nu<1/2$ and $|\lambda|=1$ \\
   \hline
  Hyperbolic Non-Automorphism & $\nu\leq1/2$ and $|\lambda|>\phi'(\eta)^{(1-2\nu)/2}$ \\
  \hline
\end{tabular}
\caption{Cyclic properties of $\lambda C_\phi$ on $\mathcal{S}_\nu$ with $\phi$ linear fractional map and $\eta$ denotes the attractive
fixed point of $\phi$}
\end{table}
\newpage
Our goal in the next two sections will be completing the previous tables by adding the recurrence of $C_\phi$ and also the recurrence of $\lambda C_\phi$ on $\mathcal{S}_\nu$. All the results about the recurrence of composition operators and its scalar multiples on $\mathcal{S}_\nu$ spaces are summarized in Table III and Table IV below.
\begin{table}[h!]
\centering
\begin{tabular}{|p{3cm}|c|c|}
  \hline
  Type of $\phi$ & Recurrence & Examples  \\
  \hline\hline
   Hyperbolic Automorphism & $\nu<1/2$ & $\frac{3z+1}{z+3}$  \\
   \hline
  Parabolic Automorphism & $\nu<1/2$ & $\frac{(1+i)z-1}{z+i-1}$  \\
  \hline
  Hyperbolic Non-Automorphism & $\nu<1/2$ & $\frac{1+z}{2}$  \\
  \hline
   Parabolic Non-Automorphism& Never & $\frac{1}{2-z}$ \\
   \hline
   Interior and Exterior& Never & $\frac{-z}{2+z}$\\
   \hline
   Interior and Boundary& Never & $\frac{z}{2-z}$  \\
   \hline
   Elliptic Irrational rotation& Always & $e^{i\sqrt{2}}z$  \\
   \hline
   Elliptic Rational rotation& Always & $e^{2i\pi/3}z$  \\
  \hline
\end{tabular}
\caption{Recurrence of $C_\phi$ on $\mathcal{S}_\nu$, $\phi$ linear fractional map.}
\end{table}

\begin{table}[h!]
\begin{tabular}{|p{3cm}|c|}
  \hline
   Type of $\phi$& Recurrence \\
   \hline\hline
   Elliptic Automorphism& $\nu\in \mathbb{R}$ and $|\lambda|=1$\\
   \hline
   Hyperbolic Automorphism& $\nu<1/2$ and $\phi'(\eta)^{(1-2\nu)/2}<|\lambda|<\phi'(\eta)^{(2\nu-1)/2}$ \\
   \hline
   Parabolic automorphism& $\nu<1/2$ and $|\lambda|=1$ \\
   \hline
  Hyperbolic Non-Automorphism & $\nu\leq1/2$ and $|\lambda|>\phi'(\eta)^{(1-2\nu)/2}$ \\
  \hline
\end{tabular}
\caption{Recurrence of $\lambda C_\phi$ on $\mathcal{S}_\nu$, $\phi$ linear fractional map and $\eta$ the attractive
fixed point of $\phi$.}
\end{table}
For the cases not appearing in the Table IV, the operator $\lambda C_\phi$ cannot be recurrent, for any $\lambda\in \mathbb{C}$.
\section{Recurrence of $C_\phi$ on $\mathcal{S}_\nu$}
A useful tool in the study of any of the dynamic properties is the following proposition known as Comparison principle (see \cite[p. 111]{Sh} and \cite{salas}).
\begin{proposition}
Let $\mathcal{E}$ be a metric space and $\mathcal{F}$ be a subspace that it self a linear metric space with a stronger topology. Suppose that $T$ is a continuous linear transformation on $\mathcal{E}$ that maps
the smaller space $\mathcal{F}$ into itself and also continuous in the topology of this space. If $T$ is cyclic on $\mathcal{F}$, then it is also cyclic on $\mathcal{E}$ and has an $\mathcal{E}$-cyclic vector that belongs to $\mathcal{F}$. Furthermore, the same is true for supercyclic and hypercyclic operators.
\end{proposition}
\begin{remark}
Observe that the comparison principle is also true for recurrent operators. If $\nu_1>\nu_2$, we have that $\mathcal{S}_{\nu_1}$ is a dense subspace of the space $\mathcal{S}_{\nu_2}$. Thus we can apply the comparison principle in the sense that, if an operator $T$ acting on $\mathcal{S}_{\nu_1}$ is recurrent, then it is recurrent on $\mathcal{S}_{\nu_2}$.
\end{remark}
\begin{theorem}\label{thm31}
Let $\phi$ be a linear fractional map on $\mathbb{D}$ with an interior fixed point $p\in \mathbb{D}$. Then $C_\phi$ is recurrent on any of the $\mathcal{S}_\nu$ spaces if and only if $\phi$ is an elliptic automorphism of $\mathbb{D}$.
\end{theorem}
\begin{proof}
If $\phi$ is an elliptic automorphism then $\phi$ is conjugate to a rotation (see \cite[Chapter 0]{Sh}). Thus, there exists a linear fractional map composition operator $S$ and a complex number $\lambda\in \mathbb{T}$ such that $C_\phi=S^{-1}C_{\phi_{\lambda}}S$, where $\phi_{\lambda}(z)=\lambda z$, $z\in \mathbb{D}$. In order to prove that $C_\phi$ is recurrent, we need to prove that $C_{\phi_{\lambda}}$ is recurrent. Since $\lambda\in \mathbb{T}$, there exists $(n_k)_{k\in \mathbb{N}}$ such that $\lambda^{n_k}\rightarrow 1$. For any $f(z) = \sum_{m\geq0}a_mz^m \in \mathcal{S}_\nu$ we have that
$$\|C_{\phi_{\lambda}}^{n_k}(f)- f\|_{\mathcal{S}_\nu}=\sum_{m\geq0}|a_m|^2 |\lambda^{m n_k}-1|^2(m+1)^{2\nu}\rightarrow 0,\ \textrm{as}\ k\rightarrow\infty.$$
Thus, every $f\in \mathcal{S}_\nu$ is a recurrent vector for $C_{\phi_\lambda}$, and so $C_{\phi_\lambda}$ is recurrent on $\mathcal{S}_\nu$. Now, if $\phi$ is not an elliptic automorphism then the Denjoy-Wolff Iteration Theorem, \cite[Proposition 1, Chapter 5]{Sh},
implies that $\phi_n(z)$ converges to $p$, for every $z\in \mathbb{D}$. Hence, if $f$ is a recurrent function of $C_\phi$ on $\mathcal{S}_\nu$, then there exists a sequence of positive integers $(n_k)_{k\in \mathbb{N}}$, such that $f\circ\phi_{n_k}\rightarrow f$ in $\mathcal{S}_\nu$. Thus, by the continuity of point evaluation functional on $S_\nu$, we have that for each $z\in \mathbb{D}$: $$f(z)=\lim_{k\rightarrow\infty}f(\phi_{n_k}(z))=f(p).$$
We conclude that the only recurrent vectors of $C_\phi$ are the constant functions.
\end{proof}
\begin{theorem}
Let $\phi$ be a parabolic automorphism of the unit disk. Then $C_\phi$ is recurrent on $\mathcal{S}_\nu$ if and only if $\nu<1/2$.
\end{theorem}
\begin{proof}
First, if $\nu<1/2$, then $C_\phi$ is hypercyclic, hence it is recurrent. Now we shall prove that this condition is necessary. Let us start by the case $\nu=1/2$. Since $\phi$ is parabolic, it has a fixed point on the unit circle. Without loss of generality we may assume this fixed point is $1$. Suppose that $C_\phi$ is recurrent, and let $f$ be a recurrent vector for $C_\phi$. Then there exists a sequence $(n_k)_{k\in \mathbb{N}}$ of positive integers such that $$C_{\phi_{n_k}}(f)\rightarrow f,$$ in $\mathcal{D}$. Since $\phi$ has no fixed point in $\mathbb{D}$, by Denjoy-Wolff Iteration Theorem, the fixed point $1$ is actually attractive. Thus, $\phi_{n_k}$ converge to $1$ uniformly on compact subsets of $\mathbb{D}$; in particular, $\phi_{n_k}(z)$ converge to $1$, for every $z\in \mathbb{D}$. Therefore we have that
$$\|C_{\phi_{n_k}}(f)-f(1)\|_{\mathcal{D}}^2=\|f\circ\phi_{n_k}-f(1)\|_{\mathcal{D}}^2
=|f\circ\phi_{n_k}(0)-f(1)|^2+
\int_{\mathbb{D}}|f'(\phi_{n_k}(z))|^2 |\phi_{n_k}'(z)|^2 dA(z)\rightarrow 0,$$
when $k\rightarrow \infty$. Thus, all recurrent vectors for $C_\phi$ are constants, and so $C_\phi$ cannot be recurrent in this case. Therefore, $C_\phi$ cannot be recurrent on $\mathcal{D}$. Hence, neither on any $S_\nu$ with $\nu>1/2$, by Comparison principle.
\end{proof}
\begin{theorem}
Let $\phi$ be a hyperbolic automorphism of the unit disk. Then $C_\phi$ is recurrent on $\mathcal{S}_\nu$ if and only if $\nu<1/2$.
\end{theorem}
\begin{proof}
If $\nu<1/2$ then the operator $C_\phi$ is recurrent. For $\nu\geq1/2$, the non-recurrence of $C_\phi$ follows exactly as in the
case of the parabolic automorphism.
\end{proof}
Now in the sequel let $\mathcal{S}_\nu^{0}$ be the space of functions obtained from the space $\mathcal{S_\nu}$ by identifying functions that differ by a constant; that is $f\in \mathcal{S}_\nu^0$ if and only if there is $g\in \mathcal{S}_\nu$ such that $f-g$ is constant. Then the space $\mathcal{S}_\nu^{0}$ is invariant under the operator $C_\phi$, and so we can consider $\tilde{C}_\phi$ the restriction of $C_\phi$ on $\mathcal{S}_\nu^{0}$.
\begin{lemma}\label{lem36}
Let $\phi$ be an analytic self-map of the unit disk. If $C_\phi$ is recurrent
on $\mathcal{S}_{\nu}$, then so is $\tilde{C}_\phi$ on $\mathcal{S}_\nu^{0}$.
\end{lemma}
\begin{proof}
We have that for every $f\in \mathcal{S}_\nu$ and $k\in \mathbb{N}$,
\begin{align*}
\big\|C_{\phi_k}f-f\big\|_{S_\nu}&\geq\big\|C_{\phi_k}f-f+f(0)-f(\phi_k(0))\big\|_{S_\nu}\\
                                 &=\big\|\tilde{C}_{\phi_k}f-f\big\|_{\mathcal{S}_\nu^{0}}.
\end{align*}

Thus, $Rec(\tilde{C}_\phi)=Rec(C_\phi)\cap \mathcal{S}_\nu^{0}$.
Assume that $C_\phi$ is recurrent, then $Rec(C_\phi)$ is dense in $\mathcal{S}_\nu$, which implies that $Rec(\tilde{C}_\phi)$ is dense in $\mathcal{S}_\nu^{0}$.
\end{proof}
\begin{theorem}
Let $\phi$ be a linear fractional map on $\mathbb{D}$ without interior fixed point. If $\phi$ is hyperbolic non-automorphism, then $C_\phi$ is recurrent on $\mathcal{S}_\nu$ if and if $\nu<1/2$.
\end{theorem}
\begin{proof}
If $\nu<1/2$ then $C_\phi$ is hypercyclic, thus it is recurrent. Now, suppose that $\nu>1/2$. Since $\phi$ has no interior fixed point, $\phi$ must have an attractive fixed point $p\in \mathbb{T}$. Thus, $\phi$ is hyperbolic non-automorphism with attractive fixed point $p\in \mathbb{T}$. Hence, the spectrum of $C_\phi$ on $S_\nu$ is given by:
$$\sigma(C_\phi)=\{\alpha\in \mathbb{C};\ |\alpha|\leq \phi'(p)^{-\gamma}\}\cup \{\phi'(p)^{k};\ k=0,1,...\},$$
where $\phi'(p)$ is the angular derivation of $\phi$ on $p$ and $\gamma=(1-2\nu)/2$, see \cite{PH}. If $\nu>1/2$, then $\gamma<0$. In this case the singleton $\{\phi'(p)\}$ is a component of the spectrum which isn't intersect $\mathbb{T}$, and so by \cite[Proposition 2.11]{costakis}, $C_\phi$ cannot be recurrent. Now, suppose that $\nu=1/2$, that is, $\mathcal{S}_\nu$ is the Dirichlet space $\mathcal{D}$. If $C_\phi$ is recurrent on the Dirichlet space, then, by Lemma \ref{lem36}, so is $\tilde{C}_\phi$ on $\mathcal{D}_0:=\mathcal{S}_{1/2}^0$. Since $\phi$ is not an automorphism, $\phi(\mathbb{D})\varsubsetneq \mathbb{D}$, and the recurrence of $C_\phi$ implies that $\phi$ is injective, then there exists a non-empty open $U\varsubsetneq \mathbb{D}$, such that $\phi(\mathbb{D})\cap U=\emptyset$. Hence, on $\mathcal{D}_0$ we have
$$\int_\mathbb{D} |f'(\phi(z))|^2 |\phi'(z)|^2 dA(z)=\int_{\phi(\mathbb{D})}|f'(z)|^2 dA(z)\leq \int_{\phi(\mathbb{D})}|f'(z)|^2 dA(z)+\int_{U}|f'(z)|^2 dA(z)<\int_{\mathbb{D}}|f'(z)|^2 dA(z).$$
Thus, $\|C_\phi\|_{\mathcal{D}_0}<1$, and therefore, $\tilde{C}_\phi$ cannot be recurrent on $\mathcal{D}_0$.
\end{proof}
\begin{theorem}
Let $\phi$ be a parabolic non-automorphism of the unit disk. Then $C_\phi$ is never recurrent in any of the weighted Dirichlet
spaces $S_\nu$.
\end{theorem}
\begin{proof}
Since $\phi$ is parabolic non-automorphism, the spectrum of $C_\phi$ on $\mathcal{S}_\nu$ is given by (see \cite{rikka}) $$\sigma(C_\phi)=\{e^{-\alpha t};\ t\geq0\}\cup\{0\},$$ thus $\{0\}$ is a component of the spectrum which does not intersect the unit circle. Hence, by \cite[Proposition 2.11]{costakis}, the operator $C_\phi$ cannot be recurrent.
\end{proof}
\section{Recurrence of $\lambda C_\phi$ on $\mathcal{S}_\nu$}
\begin{theorem}
Let $\phi$ be a holomorphic self map of $\mathbb{D}$ with an interior fixed point. Then the operator $\lambda C_\phi$ is recurrent on $\mathcal{S}_\nu$ if and only if $\phi$ is an elliptic automorphism and $\lambda\in \mathbb{T}$.
\end{theorem}
\begin{proof}
Assume that $\phi$ is an elliptic automorphism and $|\lambda|=1$. Then $C_\phi$ is recurrent by Theorem \ref{thm31}, and since $|\lambda|=1$, then $\lambda C_\phi$ is recurrent. If $\phi$ is not an elliptic automorphism, then $C_\phi$ is not recurrent, which implies that $\lambda C_\phi$ is not recurrent as well. Now, if $|\lambda|\neq1$ and $\phi$ is not an elliptic automorphism. Let $p\in \mathbb{D}$ be the fixed point of $\phi$. Suppose that $\lambda C_{\phi}$ is recurrent and let $f$ be a recurrent vector. Then there exists a sequence $(n_k)_{k\in \mathbb{N}}$ of positive integers such that
$$\lambda^{n_k}f\circ\phi_{n_k}\rightarrow f,$$ in $\mathcal{S}_\nu$. Thus, $$f(p)=\lim_{k\rightarrow\infty}\lambda^{n_k}(C_{\phi_{n_k}}f)(p)=
\lim_{k}\lambda^{n_k}f(\phi_{n_k}(p)).$$ Now, observe that $f(\phi_{n_k}(p))\rightarrow f(p)$ as $k\rightarrow \infty$. Therefore,
if $|\lambda|<1$, then $f(p)=0$ for every $z\in \mathbb{D}$. If $|\lambda|>1$, then $f(p)$ is not even defined, unless $f(p)=0$. But $f(p)$ cannot be zero for every recurrent vector, because the set of recurrent vectors is a dense subset.
\end{proof}
\begin{theorem}\label{thm41}
Let $\phi$ be a hyperbolic non-automorphism and $\eta$ its boundary fixed point. Then $\lambda C_\phi$ is recurrent on $\mathcal{S}_\nu$ if and only if $|\lambda|>\phi'(\eta)^{(1-2\nu)/2}$ and $\nu\leq1/2$. In particular, we have: For the Dirichlet space ($\nu=1/2$) the operator $\lambda C_\phi$ is recurrent if and only if $|\lambda|>1$.
\end{theorem}
To prove Theorem \ref{thm41}, we need the following lemma and proposition:
\begin{lemma}\label{lem42}
Let $T$ be a bounded operator on a separable Hilbert space $\mathcal{H}$. Suppose that for each non zero vector $g\in \mathcal{H}$ such that the set of complex numbers $$A:=\{\langle T f,g\rangle;\ f\in \mathcal{H}\}$$ is not dense in $\mathbb{C}$. Then $T$ is not recurrent.
\end{lemma}
\begin{proof}
Every non zero vector $g\in \mathcal{H}$ defines a linear functional $\psi:f\rightarrow \langle f, g\rangle$ whose range is $\mathbb{C}$. Suppose that $T$ is recurrent, then its range is dense in $\mathcal{H}$. Thus, $$\mathbb{C}=\psi(\mathcal{H})=\psi(\overline{\textrm{Im}(T)})\subset \overline{\psi(\textrm{Im}(T))}=\overline{A},$$ which implies that the set $A$ is dense in $\mathbb{C}$, a contradiction.
\end{proof}
\begin{proposition}\cite[Proposition 2.16]{EM}\label{prop44}
If $\nu<1/2$, then for each $f(z)=\sum_{n=0}^{\infty}a_n z^n\in \mathcal{S}_\nu$ there is a constant $C>0$ such that $$|f'(z)|\leq\frac{C}{(1-|z|^2)^{(3-2\nu)/2}}\ (z\in \mathbb{D}).$$
\end{proposition}
Now we prove Theorem \ref{thm41}.
\begin{proof}[Proof of Theorem \ref{thm41}]
First, if $\nu\leq1/2$ and $|\lambda|>\phi'(\eta)^{(1-2\nu)/2}$, then $C_\phi$ is hypercyclic on $\mathcal{S}_\nu$, thus it is recurrent on $\mathcal{S}_\nu$. Now we prove that the conditions are necessary. Since the recurrence is invariant under similarity, we may suppose that the boundary fixed point is $1$.

Suppose that $\nu>1/2$. Then the reproducing kernel at $1$ $$K_1(z)=\sum_{n=0}^{\infty}\frac{z^n}{(n+1)^{2\nu}}$$ is in $\mathcal{S}_\nu$. Furthermore, for any $f\in \mathcal{S}_\nu$ we have $$\langle\overline{\lambda}C_\phi^* K_1,f\rangle=\overline{\lambda}\langle K_1,C_\phi f\rangle=\overline{\lambda}\langle K_1,f\circ\phi\rangle=\overline{\lambda}f(\phi(1))=\overline{\lambda}f(1)=
\langle\overline{\lambda}K_1,f\rangle.$$ Then $\overline{\lambda}$ is an eigenvalue of $\overline{\lambda}C_\phi^*$. Thus, if $|\lambda|\neq1$, the operator $\lambda C_\phi$ is not recurrent, since the point spectrum of its adjoint $\overline{\lambda}C_\phi^*$, $\sigma_p(\overline{\lambda}C_\phi^*)$ must be in the unit circle by \cite[Proposition 2.15]{costakis}. If $|\lambda|=1$, also in this case $\lambda C_\phi$ is not recurrent, since $C_\phi$ is not recurrent on $\mathcal{S}_\nu$.

Now, if $\nu=1/2$, in this case $\mathcal{S}_\nu$ will be the Dirichlet space $\mathcal{D}$. Assume that $\lambda C_\phi$ is recurrent on $\mathcal{D}$, for some $|\lambda|\leq1$, then so is $\lambda \tilde{C}_{\phi}$ on $\mathcal{D}_0$. But on $\mathcal{D}_0$ we have $$\int_\mathbb{D}|\lambda f'(\phi(z))|^2|\phi'(z)|^2 dA(z)=|\lambda|^2\int_{\phi(\mathbb{D})}|f'(z)|^2dA(z)<\int_\mathbb{D}|f'(z)|^2dA(z).$$
Thus $\|\lambda C_\phi\|_{D_0}<1$, and therefore, $\lambda \tilde{C}_\phi$ cannot be recurrent on $\mathcal{D}_0$, a contradiction.
It remains to prove that if $|\lambda|\leq\mu^{(1-2\nu)/2}$ and $\nu<1/2$, then the $\lambda C_\phi$ cannot be recurrent on $\mathcal{S}_\nu$. The growth estimate for the derivative in Proposition \ref{prop44} provides a constant $C$ such that the first inequality below holds. Thus we have that
\begin{align*}
|\langle \lambda C_\phi f,z \rangle| &= 2^{2\nu}|\lambda||(f\circ\phi)'(0)| \\
           &= 2^{2\nu}|\lambda||f'(\phi(0))||\phi'(0)| \\
           &= 2^{2\nu}|\lambda||f'(1-\mu)|\mu\\
           &\leq\frac{2^{2\nu}C|\lambda|\mu}{(1-|1-\mu|^2)^{(3-2\nu)/2}}\\
           &\leq\frac{2^{2\nu}C \mu^{(3-2\nu)/2}}{(2\mu-\mu^2)^{(3-2\nu)/2}}\\
           &=\frac{2^{2\nu}C}{2-\mu}.
\end{align*}
By applying Lemma \ref{lem42} we see that $\lambda C_\phi$ is not recurrent.
\end{proof}
\begin{theorem}
Let $\phi$ be a parabolic automorphism of the unit disk. Then the operator $\lambda C_\phi$ is recurrent on $\mathcal{S}_\nu$ if and only if $\nu<1/2$ and $|\lambda|=1$.
\end{theorem}
\begin{proof}
First if $\nu<1/2$ and $|\lambda|=1$, then $\lambda C_\phi$ is hypercyclic on $\mathcal{S}_\nu$, which implies that it is recurrent.
Now we prove that these conditions are necessary.
Let us first examine the case $\nu=1/2$. Suppose that $|\lambda|\leq1$.
Let $f$ be a recurrent vector for $\lambda C_\phi$, then there is a sequence $(n_k)_{k\in \mathbb{N}}$ of positive integers such that $\lambda^{n_k}C_\phi^{n_k}f\rightarrow f$ in $\mathcal{D}$.
Since $\phi$ is parabolic automorphism with fixed point $1$.
If $|\lambda|>1$, then, by what we have just proved, $\lambda^{-1}C_{\phi_{-1}}$ is not recurrent.
Now, an invertible operator is recurrent if and only if its inverse is.
Consequently, $\lambda C_\phi$ is not recurrent on $\mathcal{D}$.\\
Now suppose that $\nu>1/2$, we distingue two cases:\\
\textbf{Case 1}: If $|\lambda|\neq1$, the reproducing kernel at $1$ in $\mathcal{S}_\nu$ is an eigenvalue for the adjoint $\overline{\lambda}C_\phi^{*}$, a contradiction.\\
\textbf{Case 2}: If $|\lambda|=1$, we know that $C_\phi$ is not recurrent on $\mathcal{S}_\nu$, then $\lambda C_\phi$ is not recurrent on $\mathcal{S}_\nu$ as well.\\
Now suppose that $\nu<1/2$. By Proposition the fixed point of $\phi$ must be on the boundary of the unit disk. We may suppose that the fixed point is $1$. Let $$\sigma(w)=\frac{i(1+w)}{1-w},\ \textrm{and}\ \psi=\sigma\circ\phi\circ\sigma^{-1}.$$ then $\sigma$ is a linear fractional map of the unit disk onto the upper half plane that takes $1$ to $\infty$, which implies $\infty$ is the only fixed point of $\psi$, and so $\psi(w)=w+a$, where $a\neq0$ and $\Re a=0$. The fact that $\Re a=0$ comes form that $\phi$ corresponds to an automorphism of the upper half plane. Thus $\phi$ satisfies the following formula
\begin{equation}
\phi(z)=\frac{(2-a)z+a}{-az+2+a}\ \textrm{with}\ a\neq0\ \textrm{and}\ \Re a=0.
\end{equation}
If $|\lambda|<1$. By the growth estimate for the derivative's proposition, there is a constant $C$ such that the inequality below valid
\begin{align*}
|\langle\lambda C_\phi f,z\rangle|&=2^{2\nu}|\lambda||f'(\phi(0))||\phi'(0)|\\
&=2^{2\nu}|\lambda|\big|f'\big(\frac{a}{2+a}\big)\big|\frac{4}{4+|a|^2}\\
&\leq2^{2\nu+2}C|\lambda|\big(1-\frac{|a|^2}{4+|a|^2}\big)^{(2\nu-3)/2}\frac{1}{4+|a|^2}\\
&=\frac{2^{4\nu-1}C|\lambda|}{(4+|a|^2)^{(2\nu-1)/2}}\\
&<\frac{2^{4\nu-1}C}{(4+|a|^2)^{(2\nu-1)/2}}.
\end{align*}
Thus, by Lemma \ref{lem42}, the operator $\lambda C_\phi$ is not recurrent. If $|\lambda|>1$, then $\lambda^{-1}C_{\phi_{-1}}$ is not recurrent and, therefore, neither is $\lambda C_\phi$.
\end{proof}
\begin{theorem}
Let $\phi$ be a hyperbolic automorphism of the unit disk and $\eta$ its attractive fixed point. Then $\lambda C_\phi$ is recurrent if and only if $\nu<1/2$ and $\phi'(\eta)^{(1-2\nu)/2}<|\lambda|<\phi'(\eta)^{(2\nu-1)/2}$.
\end{theorem}
\begin{proof}
For $\nu\geq1/2$, the non-recurrence of $\lambda C_\phi$ follows exactly as in the case of the parabolic automorphism.\\
Now, suppose that $\nu<1/2$. We need an expression of $\phi$. Without loss of generality, we may suppose that $\phi$ has $-1$ and $1$ as its fixed points. Moreover, we may assume that $1$ is the attractive fixed point. In order to compute $\phi$ explicitly, we use the change of variables $$\sigma(z)=\frac{i(1-z)}{1+z}$$ that sends the unit disk onto the upper half-plane, the fixed points $1$ and $-1$ to 0 and $\infty$, respectively, and $\phi$ to the contraction map $\phi(w)=\mu w$, where $0<\mu<1$. Coming back to the unit disk we have $$\phi(z)=\frac{(1+\mu)z+1-\mu}{(1-\mu)z+1+\mu}\ \textrm{with}\ 0<\mu<1.$$
Observe that the derivative at the attractive fixed point is $\phi'(1)=\mu$.\\
Now, for any $f\in \mathcal{S}_\nu$ we have the following estimate for some constant $C$,
\begin{align*}
|\langle \lambda C_\phi f,z\rangle|&=2^{2\nu}|\lambda||f'(\phi(0))||\phi'(0)|\\
&=2^{2\nu}|\lambda|\big|f'\big(\frac{1-\mu}{1+\mu}\big)\big|\frac{4\mu}{(1+\mu)^2}\\
&\frac{\leq2^{2\nu}C|\lambda|\mu^{(2\nu-1)/2}}{(1+\mu)^{2\nu-1}},
\end{align*}
that remain bounded for $|\lambda|\mu^{(2\nu-1)/2}\leq1$. Therefor, if $\lambda C_\phi$ is recurrent, then $|\lambda|>\mu^{(1-2\nu)/2}$. In addition, the inverse operator $\lambda^{-1}C_{\phi_{-1}}$ must also be recurrent. The attractive fixed point of $\phi_{-1}$ is $-1$ and $\phi'_{-1}(-1)=\mu$. Therefore, we must also have $|\lambda^{-1}|>\mu^{(1-2\nu)/2}$. Thus the conditions on $\lambda$ are necessary for $\lambda C_\phi$ to be recurrent.
\end{proof}
\begin{theorem}
Let $\phi$ be a parabolic non-automorphism of the unit disk. Then the operator $\lambda C_\phi$ is never recurrent on any $\mathcal{S}_\nu$.
\end{theorem}
\begin{proof}
Since the singleton $\{0\}$ is a component of the spectrum of $\lambda C_\phi$, but every component of the spectrum of recurrent operator must intersect the unit circle, it implies that $\lambda C_\phi$ is not recurrent.
\end{proof}

\end{document}